\documentclass{article}[12pt]
\usepackage{a4}
\usepackage[english]{babel}
\usepackage[T1]{fontenc}
\usepackage[all]{xy}
\usepackage{amssymb}
\usepackage{amsmath}
\usepackage{amsthm}

\newtheorem{thm}{Theorem}[section]
\newtheorem{prop}[thm]{Proposition}
\newtheorem{lem}[thm]{Lemma}
\newtheorem{coro}[thm]{Corollary}
\newtheorem{defn}[thm]{Definition}
\newtheorem{expl}[thm]{Example}
\newtheorem{rmq}[thm]{Remark} 

\newtheorem{conj}[thm]{Conjecture}
\numberwithin{equation}{section}   

\newcommand{\rdbb}{\mathbb{R}^d}
\newcommand{\lra}{\longrightarrow}

\newcommand{\ebarrtext}{\overline{\mbox{Emb}}_c(\mathbb{R}, \mathbb{R}^d)}

\newcommand{\rbb}{\mathbb{R}}
\newcommand{\kdpt}{\mathcal{K}_d^{\bullet}}

\newcommand{\embthm}{\overline{\emph{Emb}}_c(M, \mathbb{R}^d)}

\title{ \textbf{The rational homology of spaces of long links}}
\date{}

\author{ Paul Arnaud Songhafouo Tsopméné} 
\begin{document}
\maketitle
\begin{abstract}
We provide a complete understanding of the rational homology of the space of long links of $m$ strands in $\mathbb{R}^d$, when $d \geq 4 $. First, we construct explicitly a cosimplicial chain complex, $L^{\bullet}_*$, whose totalization is quasi-isomorphic to the singular chain complex of the space of long links. Next we show, using the fact that the Bousfield-Kan spectral sequence associated to $L^{\bullet}_*$ collapses at the $E^2$ page, that the homology Bousfield-Kan spectral sequence associated to the Munson-Voli\'c cosimplicial model for the space of long links collapses at the  $E^2$ page rationally, solving a conjecture of Munson-Voli\'c. Our method enables us also to determine the rational homology of high dimensional analogues of spaces of long links. Our last result states that the radius of convergence of the Poincar\'e series for the space of long links (modulo immersions) tends to zero as $m$ goes to the infinity. 
\end{abstract}

\section{Introduction} 

A \textit{long link of $m$ strands} in $\mathbb{R}^d$, $d \geq 3$, is a smooth embedding $\coprod_{1}^m \mathbb{R} \hookrightarrow \mathbb{R}^d$ of $m$ copies of $\mathbb{R}$ inside $\mathbb{R}^d$, which coincides outside a compact set with a fixed standard embedding. Such embedding is said to be compactly supported. We denote by $\mbox{Emb}_c(\coprod_1^m \mathbb{R}, \mathbb{R}^d)$ or simply by $\mathcal{L}_m^d$ the space of long links of $m$ strands, and define the space $\mbox{Imm}_c(\coprod_1^m \mathbb{R}, \mathbb{R}^d)$ of long immersions of $m$ strands analogously. It is clear that there is an inclusion $\mbox{Emb}_c(\coprod_1^m \mathbb{R}, \mathbb{R}^d) \hookrightarrow \mbox{Imm}_c(\coprod_1^m \mathbb{R}, \mathbb{R}^d)$, whose homotopy fiber is called the \textit{space of long links modulo immersions}, and is denoted by $\overline{\mbox{Emb}}_c(\coprod_{i=1}^m \mathbb{R}, \mathbb{R}^d)$ or simply by $\overline{\mathcal{L}}_m^d$. This paper is devoted to the study of the latter space. More precisely, we completely determine the rational homology of $\overline{\mathcal{L}}_m^d$. In fact, we prove that the homology  Bousfield-Kan spectral sequence  associated to the Munson-Voli\'c cosimplicial model for $\overline{\mathcal{L}}_m^d$ collapses at the $E^2$ page rationally.  

 Let $\mbox{Conf}(k, \mathbb{R}^d)$ denote the space of configuration of $k$ points in $\rdbb$.  We will construct an explicit cosimplicial chain complex $L^{\bullet}_*$, where
$$L^p_*=H_*(\mbox{Conf}(mp, \rdbb); \mathbb{Q}).$$ 
Let $B_d$ denote the little $d$-disks operad, and let $s^{-p}$ be the suspension functor of degree $-p$. Define the totalization $\mbox{Tot} L^{\bullet}_*$ to be
$$\mbox{Tot} L^{\bullet}_*= \underset{p \geq 0}{\bigoplus} (s^{-p} L^p_*),$$
where the differential is the alternate sum of cofaces. We will show in this introduction that the homology of this totalization can be interpreted by the $\vee_{i=1}^mS^1$-homology of $H_*(B_d; \mathbb{Q})$. 

Our first result says that the cosimplicial chain complex $L_*^{\bullet}$ gives a cosimplicial model for the singular chain complex of the space of long links. 

\begin{thm} \label{thm1} For $d \geq 4$, the totalization of $L_*^{\bullet}$ is quasi-isomorphic to the chain complex of the space of long links of $m$ strands in $\mathbb{R}^d$. That is,
$$\emph{Tot}L_*^{\bullet} \simeq C_*(\overline{\mathcal{L}}_m^d) \otimes \mathbb{Q}.$$
\end{thm}

The following corollary is an immediate consequence of Theorem~\ref{thm1}.

\begin{coro} \label{coro1}
For $d \geq 4$, the rational homology of the space of long links of $m$ strands is isomorphic to the $\vee_{i=1}^mS^1$-homology of $H_*(B_d; \mathbb{Q})$. That is,
$$H_*(\overline{\mathcal{L}}_m^d; \mathbb{Q}) \cong HH^{\vee_{i=1}^m S^1}(H_*(B_d; \mathbb{Q})).$$
\end{coro}
We now explain what we mean by the \textit{$\vee_{i=1}^mS^1$-homology of $H_*(B_d; \mathbb{Q})$}. Let $\Gamma$ denote the category of finite pointed sets whose morphisms are maps preserving the base point. If $X_{\bullet}$ is a simplicial object in $\Gamma$, and if $A$ is a contravariant functor from $\Gamma$ to chain complexes, then the composite $A(X_{\bullet})\colon \Delta \longrightarrow  \mbox{Ch}_*$ yields  a cosimplicial chain complex, and the homology of its totalization, denoted by $HH^{X}(A)$, is called   the \textit{$X$\textit{-homology of} $A$}. Here, the simplicial model (the one we construct at the beginning of Section~\ref{section4}) $(\vee_{i=1}^mS^1)_{\bullet}$ of the wedge of $m$ copies of the circle is viewed as a simplicial object in $\Gamma$, while  the homology $H_*(B_d; \mathbb{Q}) \colon \Gamma \longrightarrow \mbox{Ch}_*$  is viewed as a contravariant functor from $\Gamma$ to chain complexes. Hence, the composite  $H_*(B_d, \mathbb{Q})((\vee_{i=1}^mS^1)_{\bullet})$ gives another way to see the cosimplicial chain complex $L^{\bullet}_*$ so that we can set
$$L_*^{\bullet}:=H_*(B_d, \mathbb{Q})((\vee_{i=1}^m S^1)_{\bullet}).$$
We now state the second and the most important result of this paper, which solves a conjecture of Munson-Voli\'c. In \cite{mun_vol09}, B. Munson and I. Voli\'c build a cosimplicial space (we denote it by $    \mathcal{K}_d^{m\bullet}$) that gives a cosimplicial model for the space $\overline{\mathcal{L}}_m^d$  of long links of $m$ strands in $\mathbb{R}^d$, when $d \geq 4$.  They also define two spectral sequences that converge respectively to the homotopy and cohomology of the space $\overline{\mathcal{L}}_m^d$ of long links. In this paper, we  look at the homology Bousfield-Kan spectral sequence associated to $\mathcal{K}^{m\bullet}_d$, which converges to the homology  $H_*(\overline{\mathcal{L}}_m^d)$ (this is a consequence of Theorem~\ref{thm1}). 

\begin{conj} \emph{[Munson-Voli\'c]} This spectral sequence collapses at the $E^2$ page rationally for $d \geq 4$.
\end{conj}

\begin{thm} \label{thm2}
For $d \geq 4$, the homology Bousfield-Kan spectral sequence associated to the Munson-Voli\'c cosimplicial model $\mathcal{K}_d^{m\bullet}$ for the space of long links of $m$ strands in $\mathbb{R}^d$ collapses at the $E^2$ page rationally. 
\end{thm}

\begin{rmq}
Our method enables us also to determine the rational homology of high dimensional analogues of spaces $\overline{\emph{Emb}}_c(\coprod_{i=1}^m \mathbb{R}^n, \mathbb{R}^d)$ of long links modulo immersions.  More precisely, as in the case of long links, we construct an explicit cosimplicial chain complex $L_*^{n\bullet}$
and we prove that it gives a cosimplicial model for the singular chain complex of $\overline{\emph{Emb}}_c(\coprod_{i=1}^m \mathbb{R}^n, \mathbb{R}^d)$. Thus we obtain Theorem~\ref{thm1-1}, Corollary~\ref{coro1-1} and Proposition~\ref{prop1-1_link}. With the same strategy, one can go further by understanding the rational homology of $\overline{\mbox{Emb}}_c(\coprod_{i=1}^m \rbb^{n_i}, \rdbb)$ for any integers $n_i \geq 1$. 
\end{rmq}

The case $m=1$ (this case corresponds to  the space of long knots (modulo immersions) $\overline{\mbox{Emb}}_c(\mathbb{R}, \mathbb{R}^d)$) was studied these last years by several authors. 
First,  Sinha constructs in \cite{sin06} a cosimplicial model $\mathcal{K}_d^{\bullet}$ of the space of long knots $\overline{\mbox{Emb}}_c(\mathbb{R}, \mathbb{R}^d)$, when $d \geq 4$. Next, Lambrechts, Turchin and Voli\'c prove \cite{lam_tur_vol10} that the homology Bousfield-Kan spectral sequence associated to $\mathcal{K}_d^{\bullet}$ collapses at the $E^2$ page rationally, when $d \geq 4$. Few years later, the author \cite{songhaf12} and Moriya \cite{moriya12} prove independently that the collapsing result still holds for $d \geq 3$, and simplify the proof of the main result of \cite{lam_tur_vol10}. Notice that, in this paper, we produce (using a completely different approach than that of \cite{lam_tur_vol10, moriya12, songhaf12}) another proof of the collapsing result. 

Other interesting results obtained in the study of  the space of long knots are the following. The author \cite{songhaf12} and Moriya \cite{moriya12} independently discover the \textit{multiplicative formality} (for $d \geq 3$) of the Kontsevich operad $\mathcal{K}_d({\bullet})$. This result has two strong consequences: the first one is immediate and says that Sinha's cosimplicial space $\mathcal{K}^{\bullet}_d$ is formal, when $d \geq 3$. The second one and most important furnishes a complete understanding of the rational homology of the space of long knots as a Gerstenhaber algebra. In fact, this second consequence states \cite[Corollary 1.6]{songhaf12} that for $d \geq 4$, the isomorphism of vector spaces between the $E^2$ page and the homology $H_*(\overline{\mbox{Emb}}_c(\mathbb{R}, \mathbb{R}^d))$ of the space of long knots (modulo immersions) respects the Gerstenhaber algebra structure.

We close this introduction with our last result, which concerns the \textit{Poincar\'e series} for the space of long links. In \cite{komawila12} Komawila and Lambrechts study the $E^2$ page of the cohomology Bousfield-Kan spectral sequence associated to the Munson-Voli\'c cosimplicial space. They show that the coefficients of the associated Euler series have an exponential growth of rate $m^{\frac{1}{d-1}} > 1$. Using now our collapsing Theorem~\ref{thm2}, we deduce the following result.

\begin{thm} \label{thm3}
For $d \geq 4$ the radius of convergence of the Poincar\'e series for the space of long links (modulo immersions) $\overline{\mathcal{L}}_m^d$ is less than or equal to $(\frac{1}{m})^{\frac{1}{d-1}}$. Therefore the Betti numbers of $\overline{\mathcal{L}}_m^d$  grow at least exponentially. \index{Betti numbers}
\end{thm} 

An  immediate consequence of Theorem~\ref{thm3} is the following corollary. 

\begin{coro} \label{coro4}
For $d \geq 4$ the radius of convergence of the Poincar\'e series for $\overline{\mathcal{L}}_m^d$ tends to $0$ as $m$ goes to $\infty$.
\end{coro} 

When $m=1$ the upper bound of Theorem~\ref{thm3} is equal to $1$, and  the following theorem, due to Turchin, gives a better upper bound in that case. 

\begin{thm} \emph{\cite{turchin10}} \label{turchin10_thm}
For $d \geq 4$ the radius of convergence \index{Poincare@Poincar\'e!radius of convergence} of the Poincar\'e series \index{Poincare@Poincar\'e!series} for the space of long knots (modulo immersions) is less than or equal to $(\frac{1}{\sqrt{2}})^{\frac{1}{d-1}}$.
\end{thm} 

Since the space of $m$ copies of long knots is a retraction up to homotopy of the space of long links, Theorem~\ref{turchin10_thm} implies that the radius of convergence of the Poincar\'e series for $\overline{\mathcal{L}}_m^d$ is less than or equal to $(\frac{1}{\sqrt{2}})^{\frac{1}{d-1}}$. Our Corollary~\ref{coro4} furnishes a better upper bound for $m$ large.

\textbf{Outline of the paper} 
\begin{enumerate}
\item[-] In Section~\ref{section2} we set up some results that we will use to prove Theorem~\ref{thm1} and Theorem~\ref{thm2}. First of all, we define a manifold $M$, and we show that (see Proposition~\ref{equivalence_prop} below) the study of the space of long links (modulo immersions) is reduced to the study of the space $\overline{\mbox{Emb}}_c(M, \mathbb{R}^d)$ of compactly supported embeddings of $M$ into $\mathbb{R}^d$.  Next we recall some results, related to the Taylor tower associated to $\overline{\mbox{Emb}}_c(\mathbb{R}^n, \mathbb{R}^d)$, obtained by Arone and Turchin in  \cite{aro_tur12}. Finally we show that similar results (Proposition~\ref{topo2_prop}, Proposition~\ref{algbr2_prop}, Proposition~\ref{prop1-1} and Proposition~\ref{prop2}) hold for the space  $\overline{\mbox{Emb}}_c(M, \mathbb{R}^d)$.

\item[-] In Section~\ref{section3} we construct an explicit cosimplicial chain complex $L_*^{\bullet}$ that gives a cosimplicial model for the singular chain complex $C_*(\overline{\mbox{Emb}}_c(M, \mathbb{R}^d); \mathbb{Q})$ (Theorem~\ref{thm1}).  To prove Theorem~\ref{thm1}, we use all the results from Section~\ref{sous_section2}, and also Proposition~\ref{iso_prop}, Lemma~\ref{equi_lem} and Theorem~\ref{equi2_lem}.

\item[-] In Section~\ref{section4} we prove Theorem~\ref{thm2}. The plan of the proof is as follows. First we build a simplicial model for the wedge of $m$ copies of the circle.  Next we  prove (Lemma~\ref{equality_lem}) that the $E^1$ pages of spectral sequences $\{E^r(L_*^{\bullet})\}_{r \geq 0}$ and $\{E^r(C_*(\mathcal{K}_d^{m\bullet}; \mathbb{Q}))\}_{r \geq 0}$ are isomorphic. Using now the fact that the spectral sequence $\{E^r(L_*^{\bullet})\}_{r \geq 0}$ collapses at the $E^2$ page (Lemma~\ref{collapses_lem}) and Theorem~\ref{thm1},  we  deduce Theorem~\ref{thm2}. 

\item[-] In Section~\ref{section4'} we show that the spectral sequence computing the rational homology of the high dimensional analogues of spaces of long links collapses at the $E^2$ page.

\item[-] In Section~\ref{section5} we show that the radius of convergence of  the Poincar\'e series for the space of long links modulo $m$ copies of the space of long knots tends to zero as $m$ goes to the infinity.  This result is obtained from Theorem~\ref{thm2} and a theorem of Komawila-Lambrechts \cite{komawila12}. 
\end{enumerate}

\textbf{Acknowledgements} 
I am grateful to my advisor Pascal Lambrechts  for his outstanding guidance on this work and also for his encouragement. I also thank Gregory Arone for his valuable help with  Proposition~\ref{equivalence_prop} (which is one of the central observations in this paper) and for helpful conversations (by emails) that we have had during the preparation of  the first draft of this paper. Obviously, I cannot forget to thank Victor Turchin for his explanation on  infinitesimal bimodules and for suggesting the statement of Proposition~\ref{iso_prop}.

\section{A compactly supported version of Goodwillie-Weiss embedding calculus for the space of long links} \label{section2}

We introduce this section with Proposition~\ref{equivalence_prop}, which allows us to reduce the study of the space $\overline{\mathcal{L}}_m^d$ of long links to the study of the space $\overline{\mbox{Emb}}_c(M, \mathbb{R}^d)$ of compactly supported embeddings of some manifold  $M$ into $\mathbb{R}^d$. Before stating that proposition, we will properly define  $M$, and the spaces $\overline{\mathcal{L}}_m^d$ and $\overline{\mbox{Emb}}_c(M, \mathbb{R}^d)$. The ground field in this section is $\mathbb{Q}$. 

Roughly speaking $M$ is the complement in $\mathbb{R}^2$ of a slightly thickening of $m+1$ copies of the interval $I=[-1, 1]$. To be more precise, let  $\{a_0, a_1, \cdots, a_m\} \subseteq I$ be the family of points defined by 
$$a_i=\frac{2i-m-1}{m+1}.$$ 
Let $0 <\epsilon < \frac{2}{m+1}$ be a fixed real number. For $0 \leq i \leq m$, define 
\[ K_i =  \left\{ \begin{array}{lll}
                   I \times [a_i, a_i+\epsilon] & \mbox{if} & 0 \leq i \leq m-1 \\
									 I \times [a_m, 1]   & \mbox{if} & i = m. 
                  \end{array} \right.       \]

\begin{defn} \label{defofm}
Let $M$ be the complement of $K = \cup_{i=0}^m K_i$ in $\mathbb{R}^2$. That is,
\begin{eqnarray}\label{manifold_eq}
M=\mathbb{R}^2 \backslash \cup_{i=0}^m K_i.
\end{eqnarray}
\end{defn}

We now define the spaces $\overline{\mathcal{L}}_m^d$ and $\overline{\mbox{Emb}}_c(M, \mathbb{R}^d)$. Let $\eta \colon \rbb^2 \hookrightarrow \rdbb$ be a fixed linear embedding defined by $\eta (t,x) = (0, \cdots, 0, t, x)$. Define $\mbox{Emb}_c(M, \rdbb)$ to be the space  of smooth embeddings $f \colon M \hookrightarrow \rdbb$ such that $f(t,x) = \eta (t,x)$ for all $(t,x) \notin ]-1,1[ \times ]-1,1[$. This space is equipped with the weak $C^{\infty}$-topology. Define also the space $\mbox{Emb}_c(\coprod_{i=1}^m \rbb, \rdbb)$ of smooth embeddings $f \colon \rbb \times \{b_1, \cdots, b_m\} \hookrightarrow \rdbb$ such that $f_i (t) = \eta (t, b_i)$ for all $|t| \geq 1$ and for all $1 \leq i \leq m$. Here $b_i = \frac{a_{i-1}+ \epsilon+a_i}{2}$ and $f_i$ is the restriction of $f$ to $\rbb \times \{b_i\}$. Similarly, we have the spaces $\mbox{Imm}_c(M, \rdbb)$ and $\mbox{Imm}_c(\coprod_{i=1}^m \rbb, \rdbb)$ of compactly supported immersions. By definitions, there are inclusions 
\[ \mbox{Emb}_c(M, \rdbb) \hookrightarrow \mbox{Imm}_c(M, \rdbb) \qquad \mbox{and} \qquad \mbox{Emb}_c(\coprod_{i=1}^m \rbb, \rdbb) \hookrightarrow \mbox{Imm}_c(\coprod_{i=1}^m \rbb, \rdbb).\]     
Let $l \in \mbox{Imm}_c(M, \rdbb)$ be a fixed immersion, and let $\tilde{l} \in \mbox{Imm}_c(\coprod_{i=1}^m \rbb, \rdbb)$ be the restriction of $l$ to $\coprod_{i=1}^m \rbb \times \{b_i\}$. 

\begin{defn} \label{definition_spaces}
\begin{enumerate}
\item[(i)]  The space $\overline{\emph{Emb}}_c(M, \rdbb)$ is the homotopy pullback of the diagram 
$$\{l\} \hookrightarrow \emph{Imm}_c(M, \rdbb) \hookleftarrow \emph{Emb}_c(M, \rdbb).$$ 
\item[(ii)] The space $\overline{\emph{Emb}}_c(\coprod_{i=1}^m \rbb, \rdbb)$ or more simply $\overline{\mathcal{L}}_m^d$ is the homotopy pullback of the diagram 
 \[\{\tilde{l}\} \hookrightarrow \emph{Imm}_c(\coprod_{i=1}^m \rbb, \rdbb) \hookleftarrow \emph{Emb}_c(\coprod_{i=1}^m \rbb, \rdbb).\]
\end{enumerate}
\end{defn}

The following proposition is a central observation in this paper. 

\begin{prop} \label{equivalence_prop}
For $d \geq 3$, the space of long links modulo immersions in $\mathbb{R}^d$ is weakly equivalent to the space of smooth compactly supported embeddings (modulo immersions) of $M$ in $\mathbb{R}^d$. That is,
$$\overline{\mathcal{L}}_m^d \simeq \overline{\emph{Emb}}_c(M, \mathbb{R}^d) .$$
\end{prop}

\begin{proof}
 For $1 \leq i \leq m$, set $A_i = I \times ]a_{i-1}+ \epsilon, a_i[$, and define  $\mbox{Emb}_c( \coprod_{i=1}^m A_i, I^d)$ to be the space of smooth embeddings $f \colon \coprod_{i=1}^m A_i \hookrightarrow I^d$ satisfying the boundary conditions
\begin{enumerate}
\item[-] $f(-1, x) = \eta (-1,x)$ and $f(1, x) = \eta (1, x)$ for all $x \in ]a_{i-1}+ \epsilon, a_i[, 1 \leq i \leq m$;
\item[-] $T_{f(-1, x)} f(A_i) = T_{f(-1,x)} \eta (A_i)$ and $T_{f(1, x)} f(A_i) = T_{f(1,x)} \eta (A_i)$ for all $1 \leq i \leq m$ and for all $x \in ]a_{i-1}+\epsilon, a_i[$. Here $T_p X$ denotes the tangent space of $X$ at $p$. 
\end{enumerate}
Define also $\mbox{Emb}_c (\coprod_{i=1}^m I, I^d)$ to be the space  of smooth embeddings $f \colon \coprod_{i=1}^m I \times \{b_i\} \hookrightarrow I^d$ satisfying similar boundary conditions as above. That is, the endpoints of $f$ and the tangent vectors at those endpoints are given by $\eta$. We define similarly the spaces $\mbox{Imm}_c( \coprod_{i=1}^m A_i, I^d)$ and $\mbox{Imm}_c( \coprod_{i=1}^m I, I^d)$. As in Definition~\ref{definition_spaces}, we have the spaces $\overline{\mbox{Emb}}_c(\coprod_{i=1}^m A_i, I^d)$ and 
$\overline{\mbox{Emb}}_c(\coprod_{i=1}^m I, I^d)$ (the fixed immersions here are respectively the restrictions of $l$ to $\coprod_{i=1}^m A_i$ and to $\coprod_{i=1}^m I \times \{b_i\}$). From the definitions, one can easily see the weak equivalences 
\[\overline{\mathcal{L}}_m^d \simeq \overline{\mbox{Emb}}_c(\coprod_{i=1}^m I, I^d) \qquad \mbox{and} \qquad  \overline{\mbox{Emb}}_c(M, \mathbb{R}^d) \simeq \overline{\mbox{Emb}}_c(\coprod_{i=1}^m A_i, I^d).\]
To end the proof, it suffices to show that there is a weak equivalence 
\begin{eqnarray} \label{weak_equivalence_emb}
 \overline{\mbox{Emb}}_c(\coprod_{i=1}^m I, I^d) \simeq \overline{\mbox{Emb}}_c(\coprod_{i=1}^m A_i, I^d). 
 \end{eqnarray}
Notice that the spaces $\mbox{Emb}_c(\coprod_{i=1}^m A_i, I^d)$ and $\mbox{Emb}_c(\coprod_{i=1}^m I, I^d)$  are not weakly equivalent because of the following reason. For an element $f$ in the first space, we have the tangent space (which is a $2$-dimensional vector space) at each point $f(t, x)$, while at each point $g(t)$, where $g$ is the restriction of $f$ to $\coprod_{i=1}^m I \times \{b_i\}$, of the second space we have a $1$-dimensional vector space as the tangent space. For the homotopy fibers $\overline{\mbox{Emb}}_c(\coprod_{i=1}^m I, I^d)$  and $\overline{\mbox{Emb}}_c(\coprod_{i=1}^m A_i, I^d)$, these tangential informations disappear because of the prescribed immersion $l$. We thus obtain (\ref{weak_equivalence_emb}).  
\end{proof}

\begin{rmq} \label{weak_equiv_rmk} The idea of Proposition~\ref{equivalence_prop} (which consits in studying  $\embthm$ instead of $\overline{\mathcal{L}}_m^d$) can be generalized to the study of the rational homology of the space $\overline{\emph{Emb}}_c(\coprod_{i=1}^m \rbb^{n_i}, \rdbb)$ of high string links (here the integers $n_i$ are not necessarily the same). For example, in the case $n_i = n$ for all $i$, which will be studied in Section~\ref{section4'}, one can define a submanifold $M_n \subseteq \rbb^{n+1}$ by  
\begin{eqnarray}\label{manifoldn_eq}
M_n =\mathbb{R}^{n+1} \backslash \cup_{i=0}^m K_{in},
\end{eqnarray}
where 
\[ K_{in} =  \left\{ \begin{array}{lll}
                   I^n \times [a_i, a_i+\epsilon] & \mbox{if} & 0 \leq i \leq m-1 \\
									 I^n \times [a_m, 1]   & \mbox{if} & i = m. 
                  \end{array} \right.       \]
As a generalization of Proposition~\ref{equivalence_prop}, one can prove the following weak equivalence
\begin{eqnarray} \label{weak_equiv_mn} 
\overline{\emph{Emb}}_c(\coprod_{i=1}^m \rbb^n, \rdbb) \simeq \overline{\emph{Emb}}_c(M_n, \rdbb).
\end{eqnarray}									
\end{rmq}

The advantage to work with $\overline{\mbox{Emb}}_c(M, \mathbb{R}^d)$ instead of $\overline{\mathcal{L}}_m^d$ is the fact that one can use the same techniques, which were developed by Arone and Turchin \cite{aro_tur12} in the study of the space $\overline{\mbox{Emb}}_c(\mathbb{R}^n, \mathbb{R}^d)$. They show that the $k$th  approximation of the Taylor tower associated to  the chain complex $C_*(\overline{\mbox{Emb}}_c(\mathbb{R}^n, \mathbb{R}^d))$, that is the   Taylor tower of the functor $V \longmapsto C_*(\overline{\mbox{Emb}}_c(V, \mathbb{R}^d))$, can be expressed in terms of morphisms between infinitesimal bimodules over the operad $C_*(B_n)$. The goal of this section is to show that we obtain similar results (for the Taylor tower of $\overline{\mbox{Emb}}_c(M, \mathbb{R}^d)$ ) as them. To state and prove our results, it is easiest to first review what is done in \cite{aro_tur12}.

\subsection{Review of the Taylor tower associated to $\overline{\mbox{Emb}}_c(\mathbb{R}^n, \mathbb{R}^d)$ } \label{sous-section2}

Here we recall some results of \cite{aro_tur12}. 

Let $\mathcal{O}(\mathbb{R}^n)$ be the poset of open subsets of $\mathbb{R}^n$. Define $\widetilde{\mathcal{O}}(\mathbb{R}^n) \subseteq \mathcal{O}(\mathbb{R}^n)$ to be the subcategory of open subsets whose complement is bounded.  Define also the category $\widetilde{\mathcal{O}}_k(\mathbb{R}^n)$ to be the subcategory of $\widetilde{\mathcal{O}}(\mathbb{R}^n)$ consisting of disjoint unions $U=U_0 \cup U_1$ such that $U_0$ is the complement of a closed ball, and $U_1$ is the disjoint union of at most $k$ open balls in $\mathbb{R}^n$. Assume that there is an inclusion $\rbb^n \hookrightarrow \rdbb$.  Then one may define a contravariant functor 
$$\overline{\mbox{Emb}}_c(-, \mathbb{R}^d) \colon \widetilde{\mathcal{O}}(\mathbb{R}^n) \longrightarrow \mbox{Top}.$$
Taking that functor as an input in a "compactly supported" version of  Goodwillie-Weiss embedding calculus \cite{wei99, good_weiss99}, one obtains Proposition~\ref{topo_prop} below, which states that the $k$th approximation  $T_k\overline{\mbox{Emb}}_c(\mathbb{R}^n, \mathbb{R}^d)$ can be expressed as the space of maps between infinitesimal bimodules over the little $n$-disks operad $B_n$. Before stating that proposition, we recall the definition of an \textit{infinitesimal bimodule} from \cite[Definition 3.8]{aro_tur12} or from \cite[Definition 4.1]{turchin10}. Recall also the following notations. By $\underset{\mathcal{O}}{\mbox{InfBim}}$, we denote the category of infinitesimal bimodules over an operad $\mathcal{O}$, and by $\underset{\mathcal{O}}{\mbox{InfBim}_{\leq k}}$ we denote its $k$th truncation. If $\mathcal{B}_1$ and $\mathcal{B}_2$ are two infinitesimal bimodules over $\mathcal{O}$, we denote by $\underset{\mathcal{O}}{\mbox{hInfBim}}(\mathcal{B}_1, \mathcal{B}_2)$ the derived object of infinitesimal bimodules morphisms from $\mathcal{B}_1$ to $\mathcal{B}_2$.  We also recall (\cite[Section 5]{aro_tur12}) that by $\mbox{sEmb}(-, \mathbb{R}^d)$, we denote the functor of standard embeddings.

\begin{prop} (\emph{\cite[Theorem 6.10]{aro_tur12}} \label{topo_prop} \emph{or} \emph{\cite[Theorem 7.1]{turchin12}}) \label{thm510}
For $ d > n$ and $k \leq \infty$, we have the weak  equivalences
$$\begin{array}{lll}
T_k\overline{\emph{Emb}}_c(\mathbb{R}^n, \mathbb{R}^d)  & \simeq & \underset{B_n}{\emph{hInfBim}_{\leq k}} (\emph{sEmb}(-, \mathbb{R}^n), B_d) \\
	                                               & \simeq & \underset{B_n}{\emph{hInfBim}_{\leq k}} (B_n, B_d).
																								 																
 \end{array}
$$
 \end{prop}

Notice that  a version of Proposition~\ref{thm510} was proved \cite{boavida_weiss} by Boavida de Brito and Weiss (they develop the details of the proof of that proposition). Notice also that  Proposition~\ref{topo_prop} admits an algebraic version obtained by 
 considering the functor
$$C_* \overline{\mbox{Emb}}_c(-, \mathbb{R}^d) \colon \widetilde{\mathcal{O}}(\mathbb{R}^n) \longrightarrow \mbox{Ch}_*$$
in which  $C_*(-)$ means the normalized singular chain complex functor. 

\begin{prop}\emph{\cite[Proposition 6.13]{aro_tur12}} \label{algbr_prop}
For $d > n$ and $k \leq \infty$ there are weak equivalences

\begin{align}
T_kC_*\overline{\emph{Emb}}_c(\mathbb{R}^n, \mathbb{R}^d)  & \simeq  \underset{C_*B_n}{\emph{hInfBim}_{\leq k}} (C_*\emph{sEmb}(-, \mathbb{R}^n), C_*B_d) \\
                                                            & \simeq  \underset{C_*B_n}{\emph{hInfBim}_{\leq k}} (C_*B_n, C_*B_d). \label{infbim_over_cbn}
\end{align}

																								 																
\end{prop}

One can express (\ref{infbim_over_cbn}) in terms of morphisms between infinitesimal bimodules over the commutative operad $\mbox{Com} = H_0(B_d)$. More precisely, one has  Proposition~\ref{combining_prop} (below) in which $C_* (S^{n-})$ is viewed as an infinitesimal bimodule over Com as follows. First of all, $S^n$ is the $n$-dimensional sphere viewed as the one-point compactification of $\mathbb{R}^n$, that is, $S^n=\mathbb{R}^n \cup \{\infty\}$ with $\infty$ as the base point. Therefore, by the second part of Example~\ref{right_Gamma_module_expl}, and by Lemma~\ref{equi_lem} it follows that 
\begin{eqnarray} \label{infbim_sn}
C_*(S^{n-}) =  \{ C_*(S^{nk}) \}_{k \geq 0} =  \left\{C_*(\underbrace{S^n \times \cdots \times S^n}_{k})\right\}_{k \geq 0}
\end{eqnarray}
is an infinitesimal bimodule over Com. Also, by the first part of Example~\ref{right_Gamma_module_expl}, and by Lemma~\ref{equi_lem}, the operad $H_*(B_d)$ is an infinitesimal bimodule over Com. 

In the following proposition, the first weak equivalence (which essentially comes from the relative formality theorem \cite[Theorem 1.4]{lam_vol} of the inclusion of operads $B_n \hookrightarrow B_d$) is proved in \cite[Proposition 7.1]{aro_tur12}, and the second one in \cite[Proposition 8.3]{aro_tur12}. Notice that Proposition~\ref{combining_prop} was first proved by Arone and Turchin \cite{aro_tur12} for $d \geq 2n+1$. Actually the codimension condition, $d \geq 2n+1$ coming from the relative formality theorem of Lambrechts-Voli\'c,  was recently improved by Turchin-Willwacher (see \cite{tur_wil14}).  They show, by a more careful analysis of Lambrechts-Voli\'c construction, that the relative formality theorem also holds for $d \geq n+2$. 

\begin{prop}\emph{\cite[Proposition 7.1 and Proposition 8.3]{aro_tur12}} \label{combining_prop} For $d \geq n+2$ and $k \leq \infty$, we have the following weak equivalences
\begin{enumerate}
\item[$\bullet$] $T_kC_*\overline{\emph{Emb}}_c(\mathbb{R}^n, \mathbb{R}^d) \simeq \underset{C_*B_n}{\emph{hInfBim}_{\leq k}} (C_*B_n, H_*(B_d; \mathbb{Q})$
\item[$\bullet$] $T_kC_*\overline{\emph{Emb}}_c(\mathbb{R}^n, \mathbb{R}^d) \simeq \underset{\emph{Com}}{\emph{hInfBim}_{\leq k}} (C_*(S^{n-}), H_*(B_d; \mathbb{Q}).$
\end{enumerate}
\end{prop}


\subsection{The Taylor tower associated to $\overline{\mbox{Emb}}_c(M, \mathbb{R}^d)$} \label{sous_section2}

The goal here is to show that similar results as those mentioned in Section~\ref{sous-section2} hold for the space $\overline{\mbox{Emb}}_c(M, \mathbb{R}^d)$, where $M$ is the manifold from Definition~\ref{defofm}.  We will prove them in a more general context. That is, instead of looking at $\overline{\mbox{Emb}}_c(M, \mathbb{R}^d)$, we will look at the space $\overline{\mbox{Emb}}_c(N, \mathbb{R}^d)$ in which $N$ denotes the complement of a compact subset of $\mathbb{R}^n$. Further in Section~\ref{section3}, we will apply (in order to proof Theorem~\ref{thm1}) the results of this section by taking $N=M$.

Let $K \subseteq [-1,1]^n$ be a compact subset with a finite number (greater than or equal to two) of connected components. Define $N$ to be the complement of  $K$ in $\rbb^n$. Define also $\overline{\mbox{Emb}}_c(N, \rdbb)$ in the same way as the space $\overline{\mbox{Emb}}_c(M, \rdbb)$ from Definition~\ref{definition_spaces}. Here, the fixed immersion is a linear embedding 
\[\eta \colon \rbb^n \hookrightarrow \rdbb, \ \mbox{with} \ \eta (x_1, \cdots, x_n) = (0, \cdots, 0, x_1, \cdots, x_n). \] 
Let $\mathcal{O}(N)$ be the category whose objects are open subsets of $N$ and morphisms are inclusions. In the general theory of Goodwillie-Weiss, to study the space $\mbox{Emb}(N, \rdbb)$ of embeddings of $N$ inside $\rdbb$, we usually use the category $\mathcal{O}(N)$ as the source category for the functor $\mbox{Emb}(-, \rdbb)$. Here we want to study the space $\overline{\mbox{Emb}}_c(N, \rdbb)$ of compactly supported embeddings (modulo immersions) of $N$ in $\rdbb$. So we need to define a suitable source category, and a suitable filtration of it.   

\begin{defn}
\begin{enumerate}
\item[$\bullet$]  Define $\widetilde{\mathcal{O}}(N)  \subseteq  \mathcal{O}(N)$ to be the subcategory of  open subsets  whose the complement in $N$ is bounded.
\item[$\bullet$]  Define $\widetilde{\mathcal{O}}_k(N)$ to be the subcategory of $\widetilde{\mathcal{O}}(N)$ consisting of $U=V \cup W$ such that
     \begin{enumerate}
		      \item $V \cap W=\emptyset$;
					\item $V$ is the complement in $\rbb^n$ of a closed ball containing $K$ in $\rbb^n$;
					\item $W$ is the disjoint union of at most $k$ open balls.
		 \end{enumerate}
\end{enumerate} 

\end{defn}

Recalling  that   we have  fixed a linear embedding  $\eta \colon \rbb^n \hookrightarrow \rdbb$, one may define the contravariant functors 
$$\overline{\mbox{Emb}}_c(-, \mathbb{R}^d) \colon \widetilde{\mathcal{O}}(N) \longrightarrow \mbox{Top} \qquad  \mbox{and} \qquad  C_*\overline{\mbox{Emb}}_c(-, \mathbb{R}^d) \colon \widetilde{\mathcal{O}}(N) \longrightarrow \mbox{Ch}_*.$$
Taking these functors as inputs in Goodwillie-Weiss embedding calculus, we have the following two propositions, which are proved in the similar way as  Proposition~\ref{topo_prop} and Proposition~\ref{algbr_prop} respectively. 

\begin{prop} \label{topo2_prop}
For $d > n$ and $k \leq \infty$ there is a weak equivalence
$$ T_k\overline{\emph{Emb}}_c(N, \mathbb{R}^d)   \simeq  \underset{B_n}{\emph{hInfBim}_{\leq k}} (\emph{sEmb}(-, N), B_d).$$
\end{prop}


\begin{prop} \label{algbr2_prop}
For $d > n$ and $k \leq \infty$ there is a weak equivalence
$$ T_k C_*\overline{\emph{Emb}}_c(N, \mathbb{R}^d)   \simeq  \underset{C_*B_n}{\emph{hInfBim}_{\leq k}} (C_*\emph{sEmb}(-, N), C_*B_d).$$
\end{prop}





Applying now the relative formality theorem, which says that for $d \geq n+2$ the inclusion $B_n \hookrightarrow B_d$ is $\rbb$-formal (see \cite[Theorem 1.4]{lam_vol} and  \cite[Theorem 1]{tur_wil14}), we obtain the following proposition. 
 
\begin{prop} \label{prop1-1} For $d \geq n+2$ and $k \leq \infty$, there is a weak equivalence
$$T_kC_*(\overline{\emph{Emb}}_c(N, \mathbb{R}^d); \mathbb{Q}) \simeq \underset{C_*{B_n}}{\emph{hInfBim}_{\leq k}}(C_*(\emph{sEmb}(-, N); \mathbb{Q}), H_*(B_d; \mathbb{Q})).$$
\end{prop}

\begin{proof} The proof is the same as that of the second assertion of Proposition 7.1 from  \cite{aro_tur12}.
\end{proof}

We end this section with a proposition, which will be a key ingredient in the proof of Theorem~\ref{thm1}. Before stating that proposition we fix some notation. Let $\widehat{N}$ be the one-point compactification of $N$. That is,  
$$\widehat{N}=N \cup \{\infty\}.$$
By $N \cup \{\infty\}$ we mean $g(N) \cup \{(0, \cdots, 0, 1)\}$, where $g \colon \rbb^n \lra S^n$ is the inverse of the stereographic projection. The space $\widehat{N}$ is a pointed topological space with  $\infty$ as the base point. As the case of $C_*(S^{n-})$ from (\ref{infbim_sn}), we have a structure of an infinitesimal bimodule over Com on $C_*(\widehat{N}^{\times -})$. 
 
\begin{prop}\label{prop2} For $d \geq n+2$ and $k \leq \infty$, there is a weak equivalence
$$T_kC_*(\overline{\emph{Emb}}_c(N, \mathbb{R}^d); \mathbb{Q}) \simeq \underset{\emph{Com}}{\emph{hInfBim}_{\leq k}}(C_*(\widehat{N}^{\times -}), H_*(B_d; \mathbb{Q})).$$
\end{prop}

\begin{proof}
The  proof is the same as that of \cite[Proposition 8.3]{aro_tur12}, except that here we will work with the categories $\widetilde{\mathcal{O}}(N)$ and $\widetilde{\mathcal{O}}_k(N)$ instead of the categories $\widetilde{\mathcal{O}}(\rbb^n)$ and $\widetilde{\mathcal{O}}_k(\rbb^n)$. Recall that the two latter categories were defined at the beginning of Section~\ref{sous-section2}. 

Let $g \colon B_n \lra \mbox{Com} = \{*\}_{p \geq 0}$ be the unique morphism of operads from $B_n$ to the topological commutative operad Com, and let $C_*(g) \colon C_*(B_n) \lra C_*(\mbox{Com}) = \mbox{Com}$ be the chain complex of $g$. The morphisms $g$ and $C_*(g)$ induce respectively pairs 
\begin{eqnarray} \label{res_ind1}
\widetilde{\mbox{ind}}  \colon \underset{B_n}{\mbox{hInfBim}} \rightleftarrows \underset{\mbox{Com}}{\mbox{hInfBim}}: \mbox{res}
\end{eqnarray}
and 
\begin{eqnarray} \label{res_ind2}
\widetilde{\mbox{ind}}  \colon \underset{C_*B_n}{\mbox{hInfBim}} \rightleftarrows \underset{\mbox{Com}}{\mbox{hInfBim}}: \mbox{res}
\end{eqnarray}
of adjoint functors. Here res is the restriction functor, and $\widetilde{\mbox{ind}}$ is the induction functor defined as follows. Let $\widetilde{\Gamma}(B_n)$ be the category whose object is a finite pointed set $(S, *)$, which is viewed as $|S|$ standard balls together with one standard antiball (an \textit{antiball} is defined to be the complement of a ball).   Morphisms in $\widetilde{\Gamma}(B_n)$  are standard embeddings, $\mbox{sEmb}((S, *), (T, *))$, preserving the base point (the antiball playing the role of the base point). The category $\widetilde{\Gamma}(B_n)$ is filtered by the categories $\widetilde{\Gamma}_{\leq k}(B_n)$, where objects in $\widetilde{\Gamma}_{\leq k}(B_n)$ consist of one antiball together with (at most) $k$ open balls.  It is proved (see \cite[Proposition 4.9]{aro_tur12} for a more general statement) that the category of infinitesimal bimodules over $B_n$ is equivalent to the category of contravariant functors from $\widetilde{\Gamma}(B_n)$ to $\mbox{Top}$. We will thus identify these two categories. Let $X \colon \widetilde{\Gamma}(B_n) \lra \mbox{Top}$ be an infinitesimal bimodule over $B_n$. The object $\widetilde{\mbox{ind}}(X)$ is defined to be  the homotopy left Kan extension of $X$ along $\widetilde{\Gamma}(g)$
\[ \xymatrix{\widetilde{\Gamma}(B_n) \ar[rrr]^{X} \ar[d]_{\widetilde{\Gamma}(g)} & & & \mbox{Top} \\
                 \widetilde{\Gamma}(\mbox{Com})=\Gamma, \ar@{.>}[rrru]_-{\ \ \ \ \ \ \widetilde{\mbox{ind}}(X)} & & &  }
\]
where $\Gamma$ is the category from Definition~\ref{defn_of_gamma}. Similarly, one can define the functor $\widetilde{\mbox{ind}}$ of (\ref{res_ind2}). By noticing that $\mbox{res}(H_*(B_d)) = H_*(B_d)$, and by using the adjunction (\ref{res_ind2}), we deduce (from Proposition~\ref{prop1-1}) that 
\[T_k C_*(\overline{\mbox{Emb}}_c(N, \mathbb{R}^d); \mathbb{Q}) \simeq \underset{\mbox{Com}}{\mbox{hInfBim}_{\leq k}}(\widetilde{\mbox{ind}}(C_*(\mbox{sEmb}(-, N); \mathbb{Q})), H_*(B_d; \mathbb{Q})).\]
To end the proof it suffices to show that the functors  $C_*(\widehat{N}^{\times -})$ and $\widetilde{\mbox{ind}}(C_*\mbox{sEmb}(-, N))$ are weakly equivalent as infinitesimal bimodules over Com. Since  the functor $\widetilde{\mbox{ind}}(C_*\mbox{sEmb}(-, N))$ is the homotopy colimit of a certain diagram, and since the singular chain functor $C_*(-)$ commutes with homotopy colimits, it suffices to prove that there is a weak equivalence 
\begin{eqnarray} \label{ind_to_n}
\widetilde{\mbox{ind}}(\mbox{sEmb}(-,N)) \simeq \widehat{N}^{\times -},
\end{eqnarray} 
holding in the  category of contravariant functors from $\Gamma_{\leq k}$ to Top (here $\Gamma_{\leq k} \subseteq \Gamma$ is the subcategory whose objects are pointed sets $(S,*)$ with $|S| \leq k$). The rest of the proof is devoted to (\ref{ind_to_n}).  Notice first that 
\begin{eqnarray} \label{weakeq_withu0}
\widetilde{\mbox{ind}}\ \mbox{sEmb}(-, N) \simeq \widetilde{\mbox{ind}} \left( \underset{V \in \widetilde{\mathcal{O}}_k(N)}{ \mbox{hocolim}} \ \mbox{sEmb}(-, V) \right) \simeq \underset{V \in \widetilde{\mathcal{O}}_k(N)}{ \mbox{hocolim}} \left(\widetilde{\mbox{ind}} \ \mbox{sEmb}(-, V) \right).
\end{eqnarray} 
Let $U \in \widetilde{\Gamma}_{\leq k}(B_n)$. Then, since $\mbox{sEmb}(-, U) \colon \widetilde{\Gamma}_{\leq k} (B_n) \lra \mbox{Top}$ is the free functor generated by $U$, it follows that $\widetilde{\mbox{ind}}\ \mbox{sEmb}(-, U)$ is the free functor generated by $\widetilde{\Gamma}(g)(U)$. That is,  
\begin{eqnarray} \label{weakeq_withu1}
\widetilde{\mbox{ind}}(\mbox{sEmb}(-,U)) \simeq \mbox{Map}_*(-, \widetilde{\Gamma}(g)(U)) \colon \Gamma_{\leq k} \lra \mbox{Top},
\end{eqnarray}
which is natural in $U$. Notice  that  $\mbox{Map}_*(-, \widetilde{\Gamma}(g)(U))$ is not weakly equivalent to the functor $\mbox{Map}_*(-, U)$ because the antiball of $U$ is not contractible. To correct this, let us define $\widehat{U}$ to be the one-point compactification of $U$, that is, $\widehat{U}= U \cup \{\infty\}$. Here the point $\infty$ is of course added to the antiball of $U$, and it is the base point of $\widehat{U}$. We now have the following weak equivalence
\begin{eqnarray} \label{weakeq_withu2}
\mbox{Map}_*(-, \widetilde{\Gamma}(g)(U)) \simeq \mbox{Map}_*(-, \widehat{U}),
\end{eqnarray}
which is also natural in $U$. Combining (\ref{weakeq_withu0}), (\ref{weakeq_withu1}) and (\ref{weakeq_withu2}), one has
\begin{eqnarray} \label{weakeq_withu3}
\widetilde{\mbox{ind}}\ \mbox{sEmb}(-, N) \simeq   \underset{V \in \widetilde{\mathcal{O}}_k(N)}{ \mbox{hocolim}} \left( \mbox{Map}_*(-, \widehat{V}) \right).
\end{eqnarray} 
Since the right hand side of (\ref{weakeq_withu3}) is weakly equivalent to $\mbox{Map}_*(-, \widehat{N}) = \widehat{N}^{\times -}$, the desired result follows. 
\end{proof}

\section{A cosimplicial model for the singular chain complex of the space of long links} \label{section3}

The goal of this section is to prove Theorem~\ref{thm1} announced in the introduction. Before doing that, we state some intermediate results. As in Section~\ref{section2}, the ground field in this section is $\mathbb{Q}$. 

Let us start with the definition of a \textit{right $\Gamma$-module}.

\begin{defn} \label{defn_of_gamma} 
\begin{enumerate}
\item[$\bullet$] We define $\Gamma$ to be the category of finite pointed sets, the morphisms being the maps preserving the base point.
\item[$\bullet$] A \emph{right $\Gamma$-module} is a contravariant functor from $\Gamma$ to chain complexes $\emph{Ch}_*$.
\end{enumerate}
\end{defn}

We denote by $\underset{\Gamma}{\mbox{Rmod}}$ the category of right $\Gamma$-modules. If $\mathcal{M}_1$ and $\mathcal{M}_2$ are two right $\Gamma$-modules, by $\underset{\Gamma}{\mbox{hRmod}}(\mathcal{M}_1, \mathcal{M}_2)$, we denote the derived chain complex of right $\Gamma$-modules morphisms from $\mathcal{M}_1$ to $\mathcal{M}_2$. Here  are two examples of right $\Gamma$-modules that we look at in this paper.

\begin{expl} \label{right_Gamma_module_expl}
\begin{enumerate}
\item[(i)] The homology $H_*(B_d) \colon \Gamma \longrightarrow \mbox{Ch}_*$ defined by the formula
$$H_*(B_d)(k_+)= H_*(B_d(k)),$$
where $k_+$ is a finite pointed set of cardinal $k+1$, is a right $\Gamma$-module.
\item[(ii)] Let $X$ be a pointed topological space. The singular chain complex $C_*(X^{\times -}) \colon \Gamma \lra \emph{Ch}_*$ defined by 
 $$C_*(X^{\times -})(k_+)=C_*(\underset{\Gamma}{\emph{hom}}(k_+, X)) \cong C_*(X^{\times k})$$
is a right $\Gamma$-module.
\end{enumerate}
\end{expl}

 We are now going  to define the cosimplicial chain complex $L_*^{\bullet}$ which appears in Theorem~\ref{thm1}. From now on if $X_{\bullet}$ is a simplicial set, we will denote by $X$ its geometric realization.  Let $(\vee_{i=1}^m S^1)_{\bullet}$ denote the  simplicial  model of the wedge $\vee_{i=1}^mS^1$ of $m$ copies of the circle, which has a unique $0$-simplex and exactly $m$ non degenerate $1$-simplices (see the beginning of Section~\ref{section4} for the construction of that model). This simplicial model is actually a simplicial object in $\Gamma$, where the base points are taken to be the $0$-simplex and its degeneracies. Hence, we have the functor $(\vee_{i=1}^m S^1)_{\bullet} \colon \Delta^{op} \lra \Gamma$, and we can therefore form the composite $H_*(B_d)((\vee_{i=1}^m S^1)_{\bullet}) \colon \Delta \longrightarrow \mbox{Ch}_*$, which yields a cosimplicial chain complex.  

\begin{defn} \label{cosimplicial_defn} The cosimplicial chain complex $L_*^{\bullet}$ is defined to be the composite $H_*(B_d)((\vee_{i=1}^m S^1)_{\bullet})$,
$$L_*^{\bullet}= H_*(B_d)((\vee_{i=1}^m S^1)_{\bullet}).$$
\end{defn}

The following proposition is known to specialists, but its proof is written nowhere in my knowledge. 

\begin{prop} \label{iso_prop}
Let $X_{\bullet} \colon \Delta \longrightarrow \Gamma$ be a simplicial object in the category $\Gamma$. Let $A \colon \Gamma \longrightarrow \emph{Ch}_* $ be a right $\Gamma$-module. Then there is a weak equivalence of chain complexes
\begin{eqnarray} \label{eq1_prop3}
\emph{Tot}A(X_{\bullet}) \simeq  \underset{\Gamma}{\emph{hRmod}}(C_*(\left|X_{\bullet}\right|^{\times -}), A).
\end{eqnarray}
\end{prop}

\begin{proof} 
We will work with a field $\mathbb{K}$ of characteristic $0$. For a set $S$ we denote by $\mathbb{K}[S]$ the vector space generated by $S$, which will be viewed as a chain complex concentrated in degree 0. \\
We begin the proof by  showing that there is an isomorphism 
\begin{eqnarray} \label{eq1_prop}
C_*\left(\left|X_{\bullet}^{\times -}\right| \right) \cong C_*(\left|X_{\bullet}\right|^{\times -}).
\end{eqnarray}
of right $\Gamma$-modules. To do that, let us consider the pair of functors
$$
\xymatrix{\Gamma  \ar@<-4pt>[rr]_{\left|X_{\bullet}^{\times -}\right|} \ar@<4pt>[rr]^{\left|X_{\bullet}\right|^{\times -}} & &  \mbox{Top.} 
           } 
$$				
It is well known \cite[Theorem 14.3]{may92}(since the simplicial set $X_{\bullet}$ is countable) that there is an isomorphism
 $$\left|X_{\bullet}\right| \times \left|X_{\bullet}\right| \stackrel{\cong}{\longleftarrow} \left|X_{\bullet} \times X_{\bullet}\right|$$
in the category of topological spaces, and  we can easily see that this isomorphism induces, for each $p_+ \in \Gamma$, an isomorphism 
$$\phi_{p_+} \colon \left|X_{\bullet}\right|^{\times p} \stackrel{\cong}{\longleftarrow} \left|X_{\bullet}^{\times p}\right|$$ 
which is natural in $p_+$. We thus get a natural isomorphism 
$\xymatrix{\phi \colon \left|X_{\bullet}\right|^{\times -} & \left|X_{\bullet}^{\times -}\right| \ar[l]_-{\cong} }$
and, therefore, the isomorphism (\ref{eq1_prop}) holds in the category of right $\Gamma$-modules. From the latter isomorphism, we deduce the following one
\begin{eqnarray}\label{eq2_prop}
\underset{\Gamma}{\mbox{hRmod}}(C_*(\left|X_{\bullet}\right|^{\times -}), A) \cong \underset{\Gamma}{\mbox{hRmod}}\left(C_* \left(\left|X_{\bullet}^{\times -}\right| \right), A \right).
\end{eqnarray}
In the second part of this proof, we are going to show, since the totalization $\mbox{Tot}(A(X_{\bullet}))$ is weakly equivalent to the homotopy limit of the $\Delta$-diagram $A(X_{\bullet})$ in chain complexes, that the right hand side of (\ref{eq2_prop}) is quasi-isomorphic to the homotopy limit of a certain $\Delta$-diagram. For the rest of this proof, the standard simplicial set $\Delta^p_{\bullet}$ will be viewed as a simplicial object in $\Gamma$, where the base point of $\Delta^p_k=\underset{\Delta}{\mbox{hom}}([k], [p])$ is taken to be the null morphism. We denote by $s\Gamma$ the category of simplicial objects in $\Gamma$, and by $N$ the Dold-Kan normalization functor.   
Let us consider the pair of contravariant functors
\begin{eqnarray} \label{eq0_prop}
\xymatrix{\Gamma  \ar@<6pt>[rrrrr]^{\underset{[p] \in \Delta^{op}}{\mbox{hocolim}}\left(\mathbb{K}[\underset{\Gamma}{\mbox{hom}}(-, X_p)]\right)} \ar@<-6pt>[rrrrr]_{C_*\left(\left|X_{\bullet}^{\times -}\right|\right)} & & & & & \mbox{Ch}_*. 
           } 
\end{eqnarray}
We want to build a natural weak equivalence between these two functors. 				
So let $r_+ \in \Gamma$ be a finite pointed set. Then the simplicial structure of $X_{\bullet}$ induces a simplicial structure on $\underset{\Gamma}{\mbox{hom}}(r_+, X_{\bullet})$. By Yoneda's lemma, we have for each $p \geq 0$ the following isomorphism
$$\underset{s\Gamma}{\mbox{hom}}(\Delta^p_{\bullet}, \underset{\Gamma}{\mbox{hom}}(r_+, X_{\bullet}))=\underset{s\Gamma}{\mbox{hom}}(\underset{\Delta}{\mbox{hom}}(\bullet, [p]), \underset{\Gamma}{\mbox{hom}}(r_+, X_{\bullet}) \cong \underset{\Gamma}{\mbox{hom}}(r_+, X_p),$$  
which implies that
\begin{eqnarray} \label{eq3_prop}
\begin{array}{lll}
\underset{[p] \in \Delta^{op}}{\mbox{hocolim}}\left(\mathbb{K}[\underset{\Gamma}{\mbox{hom}}(r_+, X_p)]\right) & \cong & \underset{[p] \in \Delta^{op}}{\mbox{hocolim}}\left(\mathbb{K}[\underset{s\Gamma}{\mbox{hom}}(\Delta^p_{\bullet}, \underset{\Gamma}{\mbox{hom}}(r_+, X_{\bullet}))]\right) \\
                & \simeq & NV_{\bullet}(r_+).
\end{array}
\end{eqnarray}
Here $V_{\bullet}(r_+)$ is the simplicial chain complex defined by $V_p(r_+)= \mathbb{K}[\underset{s\Gamma}{\mbox{hom}}(\Delta^p_{\bullet}, \underset{\Gamma}{\mbox{hom}}(r_+, X_{\bullet}))]$. Notice that the isomorphism and the weak equivalence of (\ref{eq3_prop}) are natural in $r_+$.\\
On the other hand, let $W_{\bullet}(r_+)$  be the simplicial chain complex defined by $W_p(r_+)=\mathbb{K}[\underset{\mbox{Top}}{\mbox{hom}}(\Delta^p, \left|X_{\bullet}^{\times r}\right|)]$. Then the associated chain complex is nothing other than the singular chain complex $C_*\left(\left|X_{\bullet}^{\times_r}\right|\right)$. Therefore, since the chain complex associated to a simplicial abelian group is quasi-isomorphic to its Dold-Kan normalization, there  is a natural quasi-isomorphism
\begin{eqnarray} \label{eq4_prop}
C_*\left(\left|X_{\bullet}^{\times r}\right|\right) \simeq NW_{\bullet}(r_+).
\end{eqnarray}
We have just defined a pair of contravariant functors
$$\xymatrix{\Gamma \ar@<4pt>[rr]^{NV_{\bullet}} \ar@<-4pt>[rr]_{NW_{\bullet}} & & \mbox{Ch}_*.}$$
Define now $\alpha_{r_+}\colon NV_{p}(r_+) \longrightarrow NW_{p}(r_+)$ by the formula $\alpha_{r_+}(f)=\left|f\right|$, where $f \colon \Delta^{p}_{\bullet} \longrightarrow \underset{\Gamma}{\mbox{hom}(r_+, X_{\bullet})}$ is a morphism in simplicial sets. It is straightforward to check that $\alpha \colon NV_{\bullet} \longrightarrow NW_{\bullet}$ is a quasi-isomorphism natural in $r_+$. This implies (with (\ref{eq3_prop}) and (\ref{eq4_prop})) that there is a quasi-isomorphism 
\begin{eqnarray} \label{eq5_prop}
\underset{[p] \in \Delta^{op}}{\mbox{hocolim}}\left(\mathbb{K}[\underset{\Gamma}{\mbox{hom}}(-, X_p)]\right) \simeq C_*\left(\left|X_{\bullet}^{\times -}\right|\right).
\end{eqnarray}
in the category of right $\Gamma$-modules. We end the proof with the following summary
$$ \begin{array}{lll}
\underset{\Gamma}{\mbox{hRmod}}(C_*(\left|X_{\bullet}\right|^{\times -}), A) & \cong & \underset{\Gamma}{\mbox{hRmod}}\left(C_* \left(\left|X_{\bullet}^{\times -}\right| \right), A \right) \qquad \mbox{by (\ref{eq2_prop})} \\
  &\simeq & \underset{\Gamma}{\mbox{hRmod}} \left(\underset{[p] \in \Delta^{op}}{\mbox{hocolim}}\left(\mathbb{K}[\underset{\Gamma}{\mbox{hom}}(-, X_p)]\right), A \right) \qquad \mbox{by (\ref{eq5_prop})}\\
	& \simeq & \underset{[p] \in \Delta}{\mbox{holim}} \left(\underset{\Gamma}{\mbox{hRmod}} \left(\mathbb{K}[\underset{\Gamma}{\mbox{hom}}(-, X_p)], A \right) \right) \\
	& \cong & \underset{[p] \in \Delta}{\mbox{holim}}(A(X_p)) \qquad \mbox{by Yoneda's lemma} \\
	& \cong & \mbox{Tot}(A(X_{\bullet})).
	\end{array}
$$

\end{proof}



Before starting the proof of Theorem~\ref{thm1}, we need to state Lemma~\ref{equi_lem} and Theorem~\ref{equi2_lem}. The first one is an immediate consequence of \cite[Propostion 4.9]{aro_tur12}.

\begin{lem} \label{equi_lem} The category of infinitesimal bimodules over the commutative operad is equivalent to the category of right $\Gamma$-modules. That is,
$$\underset{\emph{Com}}{\emph{InfBim}} \cong \underset{\Gamma}{\emph{Rmod}}.$$
\end{lem}

Recalling the definition of $M_n$ from (\ref{manifoldn_eq}), one has the following result, which is proved  using  Goodwillie-Weiss techniques for embedding calculus \cite{weiss04, good_weiss99}. 

\begin{thm} \label{equi2_lem} For $d \geq 2n+2$, there is a weak equivalence 
$$C_*(\overline{\emph{Emb}}_c(M_n, \mathbb{R}^d); \mathbb{Q}) \simeq T_{\infty}C_*(\overline{\emph{Emb}}_c(M_n, \mathbb{R}^d); \mathbb{Q})$$
\end{thm}

\begin{proof}
There is a weak equivalence 
\begin{eqnarray} \label{weak_equiv}
C_*(\overline{\mbox{Emb}}_c(M_n, \mathbb{R}^d); \mathbb{Q}) \simeq T_{\infty}C_*(\overline{\mbox{Emb}}_c(M_n, \mathbb{R}^d); \mathbb{Q})
\end{eqnarray}
 for $d \geq  2 \mbox{dim}(M_n)+2 = 2n+4$. Since the important quantity regarding the source manifold is not the actual dimension, but the homotopy dimension, it follows that (\ref{weak_equiv}) also holds for $d \geq  2n + 2$ (notice that the manifold $M_n$ is essentially a (m+1)-punctured euclidean space $\rbb^{n+1}$. Therefore, its homotopy dimension is $n$). This is the relative (or boundary) stronger formulation of the excision estimates of Goodwillie-Klein. 
\end{proof}

The following remark is straightforward.

\begin{rmq} \label{wedge_rmq}
 The one-point compactification, denoted $\widehat{M} = M \cup \{\infty\}$ (for the meaning of "one-point compactification" see the paragraph just before Proposition~\ref{prop2}), of $M$ is weakly equivalent to the wedge of $m$ copies of the circle. That is,
\begin{eqnarray} \label{mwedge_circle}
\widehat{M} \simeq \vee_{i=1}^mS^1.
\end{eqnarray}
\end{rmq}

We are now ready to prove Theorem~\ref{thm1}, which states that the cosimplicial chain complex $L_*^{\bullet}$ defined above (see Definition~\ref{cosimplicial_defn}) is a cosimplicial model for the singular chain complex of the space $\overline{\mathcal{L}}_m^d$ of long links of $m$ strands in $\mathbb{R}^d$. 

\begin{proof}[\emph{\textbf{Proof of Theorem~\ref{thm1}}}] For $d \geq 4$, we have the following weak equivalences
$$\begin{array}{lll}
C_*(\overline{\mathcal{L}}_m^d) \otimes \mathbb{Q}  & \simeq & C_*(\overline{\mbox{Emb}}_c(M, \mathbb{R}^d); \mathbb{Q}) \qquad \mbox{by Proposition~\ref{equivalence_prop}} \\
	                                               & \simeq & T_{\infty}C_*(\overline{\mbox{Emb}}_c(M, \mathbb{R}^d); \mathbb{Q}) \qquad \mbox{by Theorem~\ref{equi2_lem} with $M_n = M_1 = M$} \\
																								& \simeq & \underset{C_*B_2}{\mbox{hInfBim}}(C_*\mbox{sEmb}(-, M), C_*B_d) \qquad \mbox{by Proposition~\ref{algbr2_prop} with $N=M$} \\
																								& \simeq & \underset{C_*B_2}{\mbox{hInfBim}}(C_*\mbox{sEmb}(-, M), H_*B_d) \qquad \mbox{by Proposition~\ref{prop1-1} with $N=M$} \\
																								& \simeq & \underset{\mbox{Com}}{\mbox{hInfBim}}(C_*(\widehat{M}^{\times -}), H_*B_d) \qquad \mbox{by Proposition~\ref{prop2} with $N=M$} \\
																								& \simeq & \underset{\Gamma}{\mbox{hRmod}}(C_*(\widehat{M}^{\times -}), H_*B_d) \qquad \mbox{by Lemma~\ref{equi_lem}}  \\
																								& \simeq & \underset{\Gamma}{\mbox{hRmod}}(C_*((\vee_{i=1}^m S^1)^{\times -}), H_*B_d) \qquad \mbox{by (\ref{mwedge_circle})} \\
																								& \simeq & \mbox{Tot} H_*B_d((\vee_{i=1}^m S^1)_{\bullet}) \qquad \mbox{by Proposition~\ref{iso_prop}} \\
																								& = & \mbox{Tot} L_*^{\bullet} \qquad \mbox{by Definition~\ref{cosimplicial_defn} of $L_*^{\bullet}$}.
																								 																
 \end{array}
$$
\end{proof}

\section{Collapsing of the homology Bousfield-Kan spectral sequence associated to the Munson-Voli\'c cosimplicial model} \label{section4}

The goal of this section is to prove Theorem~\ref{thm2} stated in the introduction. We begin by giving the simplicial model of the wedge of $m$ copies of the circle. Next we state and prove a crucial Lemma~\ref{equality_lem}. Finally we prove Theorem~\ref{thm2}. As in the previous sections, the ground field here is $\mathbb{Q}$. 

Let $\Delta^1_{\bullet}$ denote the simplicial model of the standard $1$-simplex $\Delta^1$, and let $\partial \Delta^1_{\bullet}$ denote its boundary.  Recall that $\Delta^1_p$ is a nondecreasing sequence of length $p+1$ on the alphabet \{0, 1\}. Define the simplicial set $S^1_{\bullet}$ to be the quotient
$$S^1_{\bullet}= \frac{\Delta^1_{\bullet}}{\partial \Delta^1_{\bullet}}.$$
It is clear that $S^1_{\bullet}$ is a simplicial model of the circle $S^1$.  Notice that each $S^1_p$ is a finite set pointed at 
$$*= \underbrace{0 \cdots 0}_{p+1} \sim \underbrace{1 \cdots 1}_{p+1},$$
and  faces and degeneracies preserve the base point. Therefore  $S^1_{\bullet}$ is a simplicial object in $\Gamma$. 
Define now the simplicial set   $(\vee_{i=1}^m S^1)_{\bullet}$ to be  the wedge of $m$ copies of the simplicial set $S^1_{\bullet}$, 
$$(\vee_{i=1}^m S^1)_{\bullet}= \vee_{i=1}^m(S^1_{\bullet}).$$
The following proposition is well known in the litterature.
\begin{prop} \label{cardinal_prop} The simplicial set $(\vee_{i=1}^m S^1)_{\bullet}$ is a simplicial model for the wedge $\vee_{i=1}^m S^1$. Moreover, for each $p \geq 0$, the finite pointed set $(\vee_{i=1}^m S^1)_{p}$ is of cardinal $mp+1$. That is,
    $$\left| (\vee_{i=1}^m S^1)_{p} \right|= mp+1.$$
\end{prop}

\begin{proof} It is straightforward to check that $(\vee_{i=1}^m S^1)_{\bullet}$ is a simplicial model of $\vee_{i=1}^m S^1$.\\
Let $p \geq 0$. Since $\left| S^1_p \right|=p+1$, by the definition of the wedge, it follows that $\left| (\vee_{i=1}^m S^1)_{p} \right|= mp+1$.
\end{proof}

Recall now some notation about spectral sequences. For a cosimplicial chain complex $C_*^{\bullet}$,  the associated total complex admits a natural filtration by the cosimplicial degree.  We denote by $\{E^r(C_*^{\bullet})\}_{r \geq 0}$ the spectral sequence induced by this filtration.\\
In the rest of this paper, we will denote by $\mathcal{K}_d^{m\bullet}$ the Munson-Voli\'c cosimplicial model \cite{mun_vol09} for the space of long links of $m$ strands. Notice that $\mathcal{K}_d^{m\bullet}$ is built in the same spirit as Sinha's cosimplicial model \cite{sin06} for the space of long knots.

\begin{lem} \label{equality_lem} For $d \geq 3$, the $E^1$ pages of spectral sequences $\{E^r(L_*^{\bullet})\}_{r \geq 0}$ and $\{E^r(C_*(\mathcal{K}_d^{m\bullet}; \mathbb{Q}))\}_{r \geq 0}$ are isomorphic. That is, 
$$\{E^r(L_*^{\bullet})\}_{r =1} \cong \{E^r(C_*(\mathcal{K}_d^{m\bullet}; \mathbb{Q}))\}_{r =1}.$$ 
\end{lem}

Before proving Lemma~\ref{equality_lem}, we recall the Com-infinitesimal bimodules structures of $H_*(B_d)$ and $H_*(\mathcal{K}_d)$ (note that $\mathcal{K}_d$ is the Kontsevich operad, which was defined and studied in \cite[Definition 4.1 and Theorem 4.5]{sin06}). First, as in the previous sections, Com is the $0$th homology group of the little $d$-disks operad $B_d$. That is,  $\mbox{Com} = H_0(B_d)$. Here $H_0(-)$ is viewed as a chain complex concentrated in degree $0$. Next  we endow the homology $H_*(B_d)$ with the Com-infinitesimal bimodule structure induced by the obvious morphism $\mbox{Com} \longrightarrow H_*(B_d)$. The homology $H_*(\mathcal{K}_d)$ is endowed with a similar Com-infinitesimal bimodule structure since  \cite[Theorem 5.10]{sin04}, which states that  the operads $B_d$ and $\mathcal{K}_d$ are weakly equivalent, implies the existence of an isomorphism $H_*(B_d) \stackrel{\cong}{\longrightarrow} H_*(\mathcal{K}_d)$.  


\begin{proof}[\emph{\textbf{Proof of Lemma~\ref{equality_lem}}}]
Since the diagram 
$$\xymatrix{H_*(B_d) \ar[r]^{\cong} & H_*(\mathcal{K}_d) \\
 H_0(B_d) \ar[u] \ar[r]^{\cong} & H_0(\mathcal{K}_d) \ar[u] 
}
$$
is commutative, it follows that the upper isomorphism holds in the category $\underset{\mbox{Com}}{\mbox{InfBim}}$. Therefore, since an infinitesimal bimodule over Com is the same thing as a right $\Gamma$-module (see Lemma~\ref{equi_lem}), the same isomorphism ($H_*(B_d) \cong H_*(\mathcal{K}_d)$)  holds in the category of right $\Gamma$-modules. This implies that the isomorphism 
$$\begin{array}{lll}
L_*^{\bullet}  & = & H_*(B_d)((\vee_{i=1}^m S^1)_{\bullet}) \\
               & \cong & H_*(\mathcal{K}_d)((\vee_{i=1}^m S^1)_{\bullet}) \\
							& = & H_*(\mathcal{K}_d^{m \bullet}) \qquad \mbox{by Proposition~\ref{cardinal_prop}}
\end{array}
$$
holds in the category of cosimplicial chain complexes, thus completing the proof.
\end{proof}

\begin{lem} \label{collapses_lem} For $d \geq 3$ the spectral sequence $\{E^r(L_*^{\bullet})\}_{r \geq 0}$ collapses at the $E^2$ page rationally.
\end{lem}

\begin{proof} By Proposition~\ref{cardinal_prop}  and Definition~\ref{cosimplicial_defn}, we have $L_*^{p}= H_*(B_d(mp))$ for each $p \geq 0$. Since the homology $H_*(B_d(mp)$ is a chain complex with $0$ differential, it follows that the vertical differential in the bicomplex associated to $L_*^{\bullet}$ is trivial. Therefore, the spectral sequence $\{E^r(L_*^{\bullet})\}_{r \geq 0}$ collapses at the $E^2$ page.
\end{proof}

We are now ready to prove Theorem~\ref{thm2}.

\begin{proof}[\emph{\textbf{Proof of Theorem~\ref{thm2}}}]
The proof follows from the following three points:
\begin{enumerate}
\item[$\bullet$] the $E^1$ pages of $\{E^r(L_*^{\bullet})\}_{r \geq 0}$ and $\{E^r(C_*(\mathcal{K}_d^{m\bullet}; \mathbb{Q}))\}_{r \geq 0}$ are isomorphic by Lemma~\ref{equality_lem};
\item[$\bullet$] for $d \geq 4$, the spectral sequences $\{E^r(L_*^{\bullet})\}_{r \geq 0}$ and $\{E^r(C_*(\mathcal{K}_d^{m\bullet}; \mathbb{Q}))\}_{r \geq 0}$ have the same abutment by Theorem~\ref{thm1};
\item[$\bullet$] the spectral sequence $\{E^r(L_*^{\bullet})\}_{r \geq 0}$ collapses at the $E^2$ page by Lemma~\ref{collapses_lem}.
\end{enumerate}
\end{proof}

\section{High dimensional analogues of spaces of long links} \label{section4'}

  
The goal of this short section is to show  that our method enables us also to get the collapsing at the $E^2$ page of the spectral sequence computing the rational homology of  the high dimensional analogues of spaces of long links.

Let us start with a definition. Let $\mbox{Emb}_c(\coprod_{i=1}^m \rbb^n, \rdbb)$ (respectively $\mbox{Imm}_c(\coprod_{i=1}^m \rbb^n, \rdbb)$) be the space of compactly supported embeddings (respectively immersions) of $\coprod_{i=1}^m \rbb^n$ inside $\rdbb$. 

\begin{defn} The \emph{high dimensional analogues of spaces of long links} is the homotopy fiber of the inclusion 
$$\emph{Emb}_c(\coprod_{i=1}^m \mathbb{R}^n, \rdbb) \hookrightarrow \emph{Imm}_c(\coprod_{i=1}^m \mathbb{R}^n, \rdbb),$$ 
and it is denoted by $\overline{\emph{Emb}}_c(\coprod_{i=1}^m \mathbb{R}^n, \rdbb)$.
\end{defn}

As in the case of long links, let us consider the cosimplicial chain complex 
$$L_*^{n\bullet}:=H_*(B_d, \mathbb{Q})((\vee_{i=1}^m S^n)_{\bullet})$$
in which $(\vee_{i=1}^m S^n)_{\bullet}$ is the simplicial model (built in the similar way as $(\vee_{i=1}^m S^1)_{\bullet}$ ) of the wedge $\vee_{i=1}^m S^n$ of $m$ copies of the $n$ dimensional sphere $S^n$. The following theorem gives a cosimplicial model for the singular chain complex  $C_*\overline{\mbox{Emb}}_c(\coprod_{i=1}^m \mathbb{R}^n, \mathbb{R}^d)$.

\begin{thm} \label{thm1-1}
For $d \geq 2n+2$ there is a weak equivalence 
$$\emph{Tot}L_*^{n\bullet} \simeq C_*(\overline{\emph{Emb}}_c(\coprod_{i=1}^m \mathbb{R}^n, \mathbb{R}^d)) \otimes \mathbb{Q}.$$
\end{thm}

\begin{proof} The proof works exactly as that of Theorem~\ref{thm1} given in Section~\ref{section3}. It suffices to
\begin{enumerate}
\item[-] use (\ref{weak_equiv_mn}) from Remark~\ref{weak_equiv_rmk},
\item[-] replace $M$ by $M_n$ (recall that the open submanifold $M_n \subseteq \mathbb{R}^{n+1}$ was defined in equation (\ref{manifoldn_eq})),
\item[-] replace $B_2$ by $B_{n+1}$, and of course $S^1$  by $S^n$ and $L^{\bullet}_*$ by $L^{n\bullet}_*$, 
\end{enumerate}
the rest remains unchanged. 
\end{proof}

The following corollary is a generalization of Corollary~\ref{coro1}. It is also an immediate consequence of Theorem~\ref{thm1-1}.

\begin{coro} \label{coro1-1}
For $d \geq 2n+2$ there is an isomorphism
$$H_*(\overline{\emph{Emb}}_c(\coprod_{i=1}^m \mathbb{R}^n, \mathbb{R}^d); \mathbb{Q}) \cong HH^{\vee_{i=1}^m S^n}(H_*(B_d; \mathbb{Q})).$$
\end{coro}

Let us consider now the spectral sequence $\{E^r(L_*^{n\bullet})\}_{r  \geq 0} $, the Bousfield-Kan spectral sequence associated to $L_*^{n \bullet}$. It is clear by Theorem~\ref{thm1-1} that it converges to the homology $H_*(\overline{\mbox{Emb}}_c(\coprod_{i=1}^m \mathbb{R}^n, \mathbb{R}^d); \mathbb{Q})$, when $d \geq 2n+2$. We can prove exactly as Lemma~\ref{collapses_lem} above that this spectral sequence collapses at the $E^2$ page, which gives the following result. 

\begin{prop} \label{prop1-1_link} For $d \geq 2n+2$, the spectral sequence $\{E^r(L_*^{n\bullet})\}_{r  \geq 0}$ computing the rational homology $H_*(\overline{\emph{Emb}}_c(\coprod_{i=1}^m \mathbb{R}^n, \mathbb{R}^d); \mathbb{Q})$  collapses at the $E^2$ page rationally.
\end{prop}

\section{Poincar\'e series for the space of long links modulo $m$ copies of long knots} \label{section5}

The aim of this section is to prove that the radius of convergence of the Poincar\'e series for the pair formed by the space of long links and the space of $m$ copies of long knots tends to $0$ as $m$ goes to the infinity.  We also state a conjecture followed by a theorem concerning the radius of convergence for that pair. Here, the abreviation $H^*$BKSS means cohomology Bousfield-Kan spectral sequence.

Let us start by defining expressions that appear in the title of the section.

\begin{defn} Let $X$ be a topological space. 
\begin{enumerate}
\item[$\bullet$] For $k \geq 0$ the  \emph{$k$th Betti number} \index{Betti numbers}, $b_k(X)$, of $X$ is the rank of its $k$th homology group $H_k(X)$.
\item[$\bullet$] The \emph{Poincar\'e} series \index{Poincare@Poincar\'e!series} of $X$, denoted by $P_X[x]$, is the series $P_X[x]= \sum_{k=0}^{\infty} b_k(X) x^k.$
\end{enumerate} 
\end{defn}

Up to now  we have denoted the space of long knots (modulo immersions) by $\ebarrtext$. For the sake of simplicity,  we will denote it here by $\mathcal{K}$. Let $\mathcal{K}^{\times m}$ denote the space of $m$ copies of long knots. Recall also the notation $\overline{\mathcal{L}}_m^d$ for the space of long links (modulo immersions) of $m$ strands in $\rdbb$. It is clear that $\mathcal{K}^{\times m}$ is a subspace of $\overline{\mathcal{L}}_m^d$.

\begin{defn}
The pair $(\overline{\mathcal{L}}_m^d, \mathcal{K}^{\times m})$ is called the \emph{space of long links modulo $m$ copies of long knots}.  
\end{defn}

In \cite{komawila12}, Komawila and Lambrechts studied the Euler series of the $E_1$ page of the $H^*$BKSS associated to the Munson-Voli\'c cosimplicial model  for the space of long links, and they obtained the following results. Recall that the Euler series \index{Euler series} associated to a bigraded vector space $V=\{V_{p,q} \}_{p,q \geq 0}$ is defined by
$$\chi(V)[x]= \sum_{q=0}^{\infty} \left( \sum_{p=0}^{\infty} (-1)^p \mbox{dim} (V_{p,q})\right) x^q.$$

\begin{thm}\emph{\cite[Theorem 5.1]{komawila12}} \label{komlam_thm}
For $d \geq 4$ the Euler series  $\chi(E_1)[x]$ of the $E_1$ page of the  $H^*$BKSS associated to $\mathcal{L}_m^{d \bullet}$ is given by 
\begin{eqnarray} \label{euler_serie_links}
\chi(E_1)[x]= \frac{1}{(1-x^{d-1})(1-2x^{d-1}) \cdots (1-mx^{d-1})}.
\end{eqnarray}
\end{thm}

The following corollary gives the  Euler series of the pair $(\overline{\mathcal{L}}_m^d, \mathcal{K}^{\times m})$.

\begin{coro} \emph{\cite{komawila12}} \label{komlam_coro}
For $d \geq 4$ the Euler series of the $E_2$ page of the $H^*$BKSS associated to the pair $(\mathcal{L}_m^{d \bullet}, (\kdpt)^{\times m})$ is given by 
\begin{eqnarray}  \label{euler_serie_links-knots}
\chi(E_2)[x] = \frac{1}{(1-x^{d-1})(1-2x^{d-1}) \cdots (1-mx^{d-1})} - \frac{1}{(1-x^{d-1})^m}.
\end{eqnarray}
\end{coro}

\begin{proof}
Recall first that  the pair $(\overline{\mathcal{L}}_m^d, \mathcal{K}^{\times m})$ admits a cosimplicial model $(\mathcal{L}_m^{d \bullet}, (\kdpt)^{\times m})$.  The second component of that cosimplicial model is just the product $(\kdpt)^{\times m}$ of $m$ copies of the Sinha cosimplicial model $\kdpt$.
The proof of the corollary comes from Theorem~\ref{komlam_thm} and the fact that the retraction (up to homotopy) $\overline{\mathcal{L}}_m^d \lra \mathcal{K}^{\times m}$ (see \cite[Section 2.1]{komawila12} for an explicit definition of that retraction) holds at the level of cosimplicial models so that we have the following isomorphism of spectral sequences
\begin{eqnarray} \label{ss_isomorphism}
\{ E_r((\overline{\mathcal{L}}_m^d, \mathcal{K}^{\times m})) \}_{r \geq 0} \cong \frac{\{E_r(\overline{\mathcal{L}}_m^d)\}_{r \geq 0}}{\{E_r(\mathcal{K}^{\times m})\}_{r \geq 0}}.
\end{eqnarray}
\end{proof}

From Corollary~\ref{komlam_coro} and Theorem~\ref{thm2}, we deduce the exponential growth of the Betti numbers of the pair $(\overline{\mathcal{L}}_m^d, \mathcal{K}^{\times m})$.

\begin{prop} \label{growth_bettinumbers}
For $d \geq 4$ the Betti numbers of the pair $(\overline{\mathcal{L}}_m^d, \mathcal{K}^{\times m})$ grow at least exponentially. 
\end{prop}

\begin{proof}
By (\ref{ss_isomorphism}) and Theorem~\ref{thm2}, 
the $H^*$BKSS of the pair $(\mathcal{L}_m^{d \bullet}, (\kdpt)^{\times m})$ collapses at the $E^2$ page. Moreover the coefficients of (\ref{euler_serie_links-knots}) have an exponential growth of rate $m^{\frac{1}{d-1}} > 1$, and by \cite[Proposition 4.5]{komawila12} the Betti numbers of the pair  $(\overline{\mathcal{L}}_m^d, \mathcal{K}^{\times m})$ have the same growth. 
\end{proof}

One can also see Proposition~\ref{growth_bettinumbers} as a consequence of a  theorem of Turchin  \cite[Theorem 17.1]{turchin10}, which states that the Betti numbers of the space $\mathcal{K}$ of long knots grow at least exponentially. Notice first that the concatenation operation endows $\overline{\mathcal{L}}_m^d$ and $\mathcal{K}^{\times m}$ with a structure, denoted $\times$,  of H-space. Let $\textbf{1} \in \mathcal{K}^{\times m}$ denote the unit, and consider the diagram
\begin{eqnarray} \label{fibration_square}
\xymatrix{F \ar[r]^{i} &  \overline{\mathcal{L}}_m^d \ar[r]^{\rho} & \mathcal{K}^{\times m} \\
F \ar[u]^{id} \ar[r]^-{g} & F \times \mathcal{K}^{\times m} \ar[u]^{\psi} \ar[r]^-{f} & \mathcal{K}^{\times m}, \ar[u]^{id} }
\end{eqnarray}
where
\begin{enumerate}
\item[-] $F$ is the fiber of $\rho$ over the unit $\textbf{1}$, $id$ is the identity map,
\item[-] the map $\psi$ is defined by $\psi(x,y)=i(x) \times s(y)$, where $s \colon \mathcal{K}^{\times m} \longrightarrow \overline{\mathcal{L}}_m^d$ is a section of $\rho$,
\item[-] the maps $g$ and $f$ are defined by $g(x)=(x, \textbf{1})$ and $f(x,y)=y$,
\item[-] the map $\rho$ is the one constructed in \cite[Section 2]{komawila12}.
\end{enumerate}
It is clear that the left square of (\ref{fibration_square}) commutes. The right square also commutes because of the following 
$$ \begin{array}{lll}
    \rho(\psi(x,y)) & = & \rho(i(x) \times s(y)) \\
		                & = & \rho(i(x)) \times \rho(s(y))\ \mbox{because $\rho$ is a morphism of H-spaces} \\
										& = & \textbf{1} \times y \ \mbox{because $s$ is a section of $\rho$} \\
										& = & f(x,y).
		\end{array}
$$
This implies that the triple $(id, \psi, id)$ is a morphism of fibrations, and therefore  the space $\overline{\mathcal{L}}_m^d$ is homeomorphic to the product $F \times \mathcal{K}^{\times m}$. We thus have the following inequality
$$\mbox{dim}(H_*(\overline{\mathcal{L}}_m^d)) > \mbox{dim}(H_*(\mathcal{K}^{\times m})).$$
Since  $\mbox{dim}(H_*(\mathcal{K}))$ grows at least exponentially \cite[Theorem 17.1]{turchin10}, it follows that $\mbox{dim}(H_*(\mathcal{K}^{\times m}))$ also grows at least exponentially, and the proposition follows.

Our proof has some consequences.  We have seen that the Betti numbers of the pair $(\overline{\mathcal{L}}_m^d, \mathcal{K}^{\times m})$ have an exponential growth of rate $m^{\frac{1}{d-1}}$. This implies the following corollary. 

\begin{coro} \label{radius_convergence_coro}
For $d \geq 4$, the radius of convergence \index{Poincare@Poincar\'e!radius of convergence} of the Poncar\'e series for the pair $(\overline{\mathcal{L}}_m^d, \mathcal{K}^{\times m})$ is less than or equal to  $(\frac{1}{m})^{\frac{1}{d-1}}$, and therefore tends to $0$ as $m$ goes to $\infty$.  
\end{coro}

Let us denote by $RC(X)$ the radius of convergence of the Poincar\'e series for a space $X$. Specially for the space of long knots, we will denote it by $R$.

\begin{rmq}
As a consequence of Theorem~\ref{turchin10_thm}
 we have the inequality $RC(\overline{\mathcal{L}}_m^d, \mathcal{K}^{\times m}) \leq (\frac{1}{\sqrt{2}})^{\frac{1}{d-1}}$. Our approach gives a better upper bound of this radius thanks to Corollary~\ref{radius_convergence_coro}. 
\end{rmq}

We end this section with a conjecture having a nice consequence.

\begin{conj} \label{rc_knots_conjecture}
The radius of convergence  of the Poincar\'e series of the space of long knots (modulo immersions)  is greater than $0$. That is, $R > 0$.
\end{conj}

Corollary~\ref{radius_convergence_coro} tells us that the radius of convergence of the Poincar\'e series for $(\overline{\mathcal{L}}_m^d, \mathcal{K}^{\times m})$ is less than or equal to  $(\frac{1}{m})^{\frac{1}{d-1}}$, but does not tell us that it is less than $R$. We therefore have the following theorem.

\begin{thm} If Conjecture~\ref{rc_knots_conjecture} is true, then for $d \geq 4$ and for $m  > \frac{1}{R^{d-1}}$ the radius of convergence of the Poincar\'e series for the pair $(\overline{\mathcal{L}}_m^d, \mathcal{K}^{\times m})$ is less than $R$. That is, $RC(\overline{\mathcal{L}}_m^d, \mathcal{K}^{\times m}) < R$.
\end{thm}

\begin{proof} 
The proof comes immediately from Corollary~\ref{radius_convergence_coro} and the hypothesis $m  > \frac{1}{R^{d-1}}$.
\end{proof}


\textsf{Université catholique de Louvain, Chemin du Cyclotron 2, B-1348 Louvain-la-Neuve, Belgique\\
Institut de Recherche en Mathématique et Physique\\}
\textit{E-mail address: arnaud.songhafouo@uclouvain.be}


\begin{thebibliography}{99}
\bibitem{aro_lam_vol07} G. Arone, P. Lambrechts and I. Voli\'c , Calculus of functors, operad formality, and rational homology of embedding spaces, \textit{Acta Math.} 199 (2007), no. 2, 153--198.
\bibitem{aro_tur12} G. Arone and V. Turchin, On the rational homology of high dimensional analogues of spaces of long knots, \textit{Geom. Topol.} vol. 18 (2014) 1261--1322. 
\bibitem{boavida_weiss} P. Boavida de Brito and M. Weiss, Manifold calculus and homotopy sheaves, \textit{Homology Homotopy Appl.} 15 (2013), no. 2, 361-383.
\bibitem{good_weiss99} T. G. Goodwillie and M. Weiss, Embeddings from the point of view of immersion theory: Part II, \textit{Geom. Topol.} 3 (1999) 103--118.
\bibitem{komawila12} G. Komawila Ngidioni and P. Lambrechts, Euler series, Stirling numbers and the growth of the homology of the space of long links,  \textit{Belgian Mathematical Society-Simon Stevin. Bulletin}. (2012) (1370--1444). 
\bibitem{lam_tur_vol10} P. Lambrechts, V. Turchin and I. Vol\'c, The rational homology of spaces of long knots in codimension > 2,  \textit{Geom. Topol.} 14 (2010) 2151--2187. 
\bibitem{lam_vol} P. Lambrechts and I. Voli\'c, Formality of the little N-disks operad, \textit{Mem. Amer. Math. Soc.} 230 (2014), no. 1079, viii+116 pp. ISBN: 978-0-8218-9212-1.  
\bibitem{may92} J. P. May, Simplicial objects in algebraic topology, \textit{University of Chicago Press}, Chicago, IL, 1992. 
\bibitem{moriya12} S. Moriya, Sinha's spectral sequence and the homotopical algebra of operads, arXiv:1210.0996.
\bibitem{mun_vol09} B. A. Munson and I. Voli\'c, Cosimplicial model for spaces of links,  \textit{J. Homotopy Relat. Struct.} 9 (2014), no. 2, 419--454.
\bibitem{pirash00} T. Pirashvili, Hodge decomposition for higher order Hochschild homology, \textit{Ann. Sci. Ecole Norm. Sup} (4) 33 (2000), no. 2, 151--179.
\bibitem{sin04} D. Sinha, Manifold-theoretic  compactifications of configuration spaces, \textit{Selecta Math.} 10 (2004), no. 3, 391-428.
\bibitem{sin06} D. Sinha, Operads and knot spaces, \textit{J. Amer. Math. Soc.}, 19(2):461--486(electronic), 2006.
\bibitem{songhaf12} P. A. Songhafouo Tsopméné, Formality of Sinha's cosimplicial model for long knots spaces and the Gerstenhaber algebra structure of homology, \textit{Algebr. Geom. Topol.} 13 (2013) 2193--2205.
\bibitem{turchin12} V. Turchin, Context-free manifold calculus and the Fulton-MacPherson operad,  \textit{Algebr. Geom. Topol.} 13 (2013), no. 3, 1243--1271.
\bibitem{turchin10} V. Turchin, Hodge-type decomposition in the homology of long knots, \textit{J. Topol.} 3 (2010), no. 3, 487--534.
\bibitem{tur_wil14} V. Turchin and T. Willwacher, Relative (non-)formality of the little cubes operads and the algebraic Schoenflies theorem (2014), arXiv:1409.0163v1. 
\bibitem{wei99} M. Weiss, Embedding from the point of view of immersion theory: Part I, \textit{Geom. Topol.} 3 (1999), 67--101.
\bibitem{weiss04} M.Weiss, Homology of spaces of smooth embeddings, \textit{Q. J. Math. 55} (2004), no. 4, 499504.

\end{thebibliography}
\end{document}